\documentclass[12pt, reqno]{amsart}  

\usepackage{amsmath,amsthm,amssymb,amsbsy,color}
\usepackage[T1]{fontenc}
\usepackage[latin9]{inputenc}
\usepackage{esint}

\usepackage[all]{xy}
\usepackage{setspace}
\allowdisplaybreaks

\newtheorem{theorem}{Theorem}[section]

\newtheorem{lemma}[theorem]{Lemma}

\theoremstyle{remark}
\newtheorem{remark}[theorem]{Remark}
\newtheorem{example}[theorem]{Example}
\numberwithin{equation}{section}

\newcommand{\Q}{{\mathbb Q}}

\title[Duality/Sum formulas for iterated integrals]
{
Duality/Sum formulas for iterated integrals 
and their application to multiple zeta values \\[+.5em] 
}
\author{Minoru Hirose, Kohei Iwaki, Nobuo Sato, Koji Tasaka}

\address[Minoru Hirose]{Multiple Zeta Research Center, Kyushu University}
\email{m-hirose@math.kyushu-u.ac.jp} 
\address[Kohei Iwaki]{Graduate School of Mathematics, Nagoya University}
\email{iwaki@math.nagoya-u.ac.jp}
\address[Nobuo Sato]{Graduate School of Mathematics, Kyoto University}
\email{saton@math.kyoto-u.ac.jp}
\address[Koji Tasaka]{Department of Information Science and Technology, Aichi Prefectural University}
\email{tasaka@ist.aichi-pu.ac.jp}

\begin{document}

\maketitle

\begin{abstract}
We investigate linear relations 
among a class of iterated integrals on the Riemann sphere 
minus four points $0,1,z$ and $\infty$.
Generalization of the duality formula and the sum formula 
for multiple zeta values to the iterated integrals are given.
\end{abstract}

\section{Introduction and main results}
\label{section:introduction}

The multiple zeta value, originally considered by Euler 
\cite{Euler}, 
is defined for positive integers $k_1,\ldots,k_r$ with $k_r\ge2$ by
\[
\zeta(k_1,\ldots,k_r)=\sum_{0<m_1<\cdots<m_r} 
\frac{1}{m_1^{k_1}\cdots m_r^{k_r}}. 
\]
These real numbers satisfy numerous linear relations 
over $\Q$ and have been studied in recent years 
by many mathematicians and physicists 
(see \cite{EF, Goncharov, IKZ, Okuda, Zagier} for example), 
but their structure is not completely understood at the time of writing. 

In this paper we present a new approach to linear relations 
among multiple zeta values through a study of linear relations 
among the iterated integrals on 
${\mathbb P}^1 \setminus\{ 0,1,z,\infty \}$
with a complex variable 
$z \in {\mathbb C}\setminus [0,1]$:
\begin{eqnarray*}
I(0;a_1,\dots,a_n;1) = \int_{0}^{1} \frac{dt_{n}}{t_n-a_n} 
\int_{0}^{t_{n}}\frac{dt_{n-1}}{t_{n-1}-a_{n-1}} 
\cdots 
\int_{0}^{t_2}\frac{dt_{1}}{t_1-a_1}
\quad 
(a_1,\dots,a_n \in \{0,1,z\}). 
\end{eqnarray*}
We assume $a_1 \ne 0$ and $a_n \ne 1$ for the convergence of the integral.
Since the iterated integral $I(0;a_1,\dots,a_n;1)$ 
is a holomorphic function of $z$, we may use analytic methods 
to obtain relations among these iterated integrals. 
In fact, a formula for the differentiation of 
$I(0;a_1,\dots,a_n;1)$ with respect to $z$ 
(Theorem 2.1; see also \cite[Lemma 3.3.30]{Erik2}) 
is useful to prove linear relations.  
Moreover, using the iterated integral expression 
\begin{equation}\label{itintex}
\zeta(k_1,\dots,k_r) = 
(-1)^r I(0;1,\overbrace{0,\dots,0}^{k_1-1},\dots,1,
\overbrace{0,\dots,0}^{k_r-1};1)
\end{equation}
of multiple zeta values due to Kontsevich and Drinfeld, 
one may obtain linear relations among multiple zeta values 
as a specialization of linear relations among 
$I(0;a_1,\dots,a_n;1)$'s.
With this approach, we rediscover the so-called duality 
formula and the sum formula satisfied by multiple zeta values 
as a specialization of the linear relation among 
the iterated integrals obtained in this paper 
(Theorems 1.1 and 1.2). Our ultimate goal is to capture 
all linear relations among the iterated integral. 
This will be developed minutely in the upcoming paper 
\cite{HiroseSato1} (see also \cite{HiroseSato2}).

Let us formulate our main results.
It is convenient to use the algebraic 
setup given by Hoffman \cite{Hoffman2} for multiple zeta values. 
Let $\mathcal{A}=\Q\langle e_0,e_1,e_z\rangle$ be the non-commutative 
polynomial algebra over $\Q$ in three indeterminate elements 
$e_0,e_1$ and $e_z$, and $\mathcal{A}^0$ be its linear subspace 
$\Q+\Q e_z +e_z\mathcal{A}e_z+e_z\mathcal{A}e_0+
e_1\mathcal{A}e_z+e_1\mathcal{A}e_0$. 
We also denote by $L$ the $\Q$-linear map that assigns 
the iterated integral $I(0; a_1,a_2,\ldots,a_n;1)$
to a word $e_{a_1}e_{a_2}\cdots e_{a_n}$ in $\mathcal{A}^0$, 
which by definition converges absolutely. 
For example, by \eqref{itintex} we have 
$L(e_1e_0^{k_1-1}\cdots e_1 e_0^{k_r-1}) 
= (-1)^{r}\zeta(k_1,\dots,k_r)$. 

Our duality formula is stated as follows.
\begin{theorem}[Duality formula] \label{thm:B-duality}
Let $\tau$ be the anti-automorphism on $\mathcal{A}$ 
(i.e., $\tau(xy) = \tau(y)\tau(x)$ holds for $x,y \in {\mathcal A}$)
given by
\[ 
\tau(e_0)=e_z-e_1,\ \tau(e_1)=e_z-e_0,\ \tau(e_z)=e_z.
\]
Then, for any $w\in \mathcal{A}^0$, we obtain
\begin{equation} \label{eq:duality-intro}
L(w-\tau(w))=0.
\end{equation}
\end{theorem}
Our derivation of the formula \eqref{eq:duality-intro} 
can be regarded as an application of Okuda's result 
(\cite[Theorem 4.2]{Okuda}). 
Thanks to the existence of a special M\"obius transformation 
on ${\mathbb P}^1$ which preserves the subset $\{0,1,z,\infty \}$, 
we could derive the equality \eqref{eq:duality-intro}.
Taking $z \to \infty$ in \eqref{eq:duality-intro}, 
we can recover the duality formula for multiple zeta values 
\cite[Section 9]{Zagier} (see also \cite{Hoffman}):
\begin{equation} \label{eq:duality-MZV}
I(0;a_1,\dots,a_{n};1) = (-1)^{n}
I(0;1-a_n,\dots,1-a_1;1)\qquad 
(a_1, \dots, a_n \in \{0,1 \})
\end{equation}
On the other hand, setting $z=-1$ in  
\eqref{eq:duality-intro},
we can also derive the Broadhurst duality 
(see \cite[eq.~(127)]{Broadhurst}, 
\cite[Example 4.2]{Okuda} and \cite[Section 4]{EF}), 
which is known as a family of $\Q$-linear relations 
among alternating Euler sums.

Our sum formula is stated as follows. 
\begin{theorem}[Sum formula] \label{thm:sum-formula}
For any integers $k, r$ satisfying $k\ge2$ and 
$k \ge r \ge 1$, we have
\begin{align}\nonumber
 & (-1)^{r}\sum_{\substack{k_{1}+\cdots+k_{r}=k\\
k_{i}\ge1,~k_{r}\ge2}}
L(e_{z}e_{0}^{k_{1}-1}e_{1}e_{0}^{k_{2}-1}\cdots e_{1}e_{0}^{k_{r}-1})\\
 & =-L(e_{1}e_{0}^{k-1})+L((e_{1}-e_{z})(e_{0}-e_{z})^{r-1}e_{0}^{k-r}). 
 \label{eq:sum-formula-intro}
\end{align}
\end{theorem}

We remark that, taking the limit $z \to 1+0$, 
the above formula is reduced to 
the following sum formula for multiple zeta values: 
For any integers $k>r>1$, we can take the limit in the both 
hand sides of \eqref{eq:sum-formula-intro} and we obtain
\begin{equation}  \label{eq:sum-MZV}
\sum_{\substack{k_1+\cdots+k_r=k\\k_i \ge1, ~k_r\ge2}} 
\zeta(k_1,\ldots,k_r) = \zeta(k).
\end{equation}
This formula was first proved by Granville \cite{Granville} and 
Zagier (see also \cite[\S 3]{Hoffman}) independently. 
Our method gives a new proof of \eqref{eq:sum-MZV}. 

Here we briefly illustrate a few examples of relations, 
which might be useful to understand the strategy of our proof.
As special cases of Theorem \ref{thm:sum-formula}, 
we will prove
\begin{align}
\label{eg1} & I(0;z,1,0;1)+I(0;1,z,0;1)+I(0;z,0,0;1)-I(0;z,z,0;1)=0
\end{align}
for the case of $(k,r)=(3,2)$, and 
\begin{align}
\label{eg2} & I(0;z,z;1)-I(0;z,0;1)-I(0;1,z;1)=0
\end{align}
for is the case $(k,r)=(2,2)$. 
We can show that the equality \eqref{eg1} can be obtained by
integrating the both hand sides of \eqref{eg2}. 
(Integration constants can be determined by considering 
the limit $z \to \infty$). 
This is a consequence of a formula describing 
the differentiation of the iterated integrals with respect to $z$ 
(Theorem \ref{thm:differential-formula}). 
In general, we can check that the differentiation 
reduces the value of $k$ or $r$ in the desired formula 
\eqref{eq:sum-formula-intro}, and hence a proof 
by the induction works perfectly. 
The differential formula plays an essential role here, 
and it is the main advantage to consider the iterated 
integral with a complex variable $z$.

This paper is organized as follows.  
In Section \ref{section:iterated-integrals} we fix notations 
and give a proof of the formula for differentiation of 
the iterated integrals with respect to $z$ 
(Theorem \ref{thm:differential-formula}). 
This differential formula can be regarded as a special case 
of \cite[Lemma 3.3.30]{Erik2}.
We give another proof of the fact.
Section \ref{section:proof} is devoted to proofs of 
our main theorems along the ideas presented above. 

\section*{Acknowledgements}
This work was partially supported by JSPS 
KAKENHI Grant Number 16H07115, 16H06337 and 16K17613.

\section{Iterated integrals with a complex variable
and the differential formula}
\label{section:iterated-integrals}


\subsection{Definitions and notations}

In this paper we assume that $z \in {\mathbb C} \setminus [0,1]$.
For any $a_1, \dots, a_n \in \{0,1,z\}$ satisfying 
$a_1 \ne 0$ and $a_n \ne 1$ we use the standard notation 
$I(0;a_{1},\ldots,a_{n};1)$ for the iterated integral: 
\begin{align} 
I(0;a_1,\dots,a_n;1) &:= 
\mathop \int_{0<t_1<t_2<\cdots<t_n<1} 
\prod_{i=1}^{n} \frac{dt_i}{t_i - a_i} \nonumber \\
&~ = \int_{0}^{1} \frac{dt_{n}}{t_n-a_n} 
\int_{0}^{t_{n}}\frac{dt_{n-1}}{t_{n-1}-a_{n-1}} 
\cdots 
\int_{0}^{t_2}\frac{dt_{1}}{t_1-a_1} . \label{eq:it-int}
\end{align}
The above iterated integral is called the
hyperlogarithm which was introduced long ago 
(e.g., \cite{Lapp-Danilevsky}), and has been studied 
by many mathematicians and physicists 
(see, \cite{Broadhurst, Erik2, Goncharov} for example). 
As a function of $z$, 
the iterated integral $I(0;a_1,\dots,a_n;1)$ is a (single-valued) 
holomorphic function on the domain ${\mathbb C} \setminus [0,1]$, 
and $\lim_{z \to \infty} I(0;a_1,\dots,a_n;1) = 0$ holds if 
$z \in \{a_1,\dots, a_n\}$. 
%
The ${\mathbb Q}$-linear 
space spanned by iterated integrals of such a form 
is denoted by
\[
{\mathcal Z}^{(z)} := \left\langle I(0;a_1,\dots,a_n;1) ~|~ 
n \ge 0,~a_i \in \{0,1,z \},~
a_1 \ne0,~ a_n\ne1 \right\rangle_{{\mathbb Q}}.
\]

Let $\mathcal{A} := {\mathbb Q}\langle e_0, e_1, e_z \rangle$
be the non-commutative polynomial algebra of 
words consisting of three letters $e_0, e_1, e_z$, and 
\[
\mathcal{A}^0 := 
\Q+\Q e_z +e_z\mathcal{A}e_z+e_z\mathcal{A}e_0+
e_1\mathcal{A}e_z+e_1\mathcal{A}e_0
\]
the subalgebra of convergent words. 
We also define a ${\mathbb Q}$-linear map 
$L : \mathcal{A}^0 \rightarrow \mathcal{Z}^{(z)}$ by 
\begin{eqnarray*}
L(1) := 1, \qquad
L(e_{a_1}e_{a_2}\cdots e_{a_n}) 
:= I(0;a_1,a_2,\dots,a_n;1).
\end{eqnarray*}

\subsection{Differential formula for the iterated integrals}
For any convergent word $w \in {\mathcal A}^{0}$, we regard 
the corresponding iterated integral $L(w)$ as a holomorphic 
function of $z$ on the domain ${\mathbb C} \setminus [0,1]$.
Here we show a formula which describes the differentiation 
of the iterated integral with respect to $z$. 
The formula is quite useful in the proof of our main result.

For the convenience, we write $a_0 = 0$ and 
$a_{n+1} = 1$ throughout this paper.
For $x,y\in\{0,1,z\}$, 
we define a ${\mathbb Q}$-linear map
$\partial_{x,y}:{\mathcal A}^{0} \rightarrow {\mathcal A}^{0}$ by 
\[
\partial_{x,y}(1) := 0, \quad 
\partial_{x,y}\left(e_{a_{1}}\cdots e_{a_{n}}\right):=
\sum_{i=1}^{n}\left(\delta_{\{a_{i},a_{i+1}\},\{x,y\}} 
- \delta_{\{a_{i-1},a_{i}\},\{x,y\}}\right)e_{a_{1}}
\cdots\widehat{e_{a_{i}}}\cdots e_{a_{n}},
\]
where 
\begin{eqnarray*}
\delta_{\{a,b \}, \{x,y \}} := 
\begin{cases} 
1 & \text{if $\{a,b \} = \{x,y \}$ as subsets of $\{0,1,z \}$} \\
0 & \text{otherwise}.
\end{cases}
\end{eqnarray*}
It is easy to check that $\partial_{x,y}$ is well-defined. 
Then we have the following theorem.
\begin{theorem}[Differential formula 
{(c.f., \cite[Lemma 3.3.30]{Erik2})}]
\label{thm:differential-formula}
For any $w\in {\mathcal A}^{0}$, we have
\begin{equation} \label{eq:differential-formula}
\frac{d}{dz}L(w) =\sum_{a\in\{0,1\}}
\frac{1}{z-a}L(\partial_{z,a}w).
\end{equation}
\end{theorem}

Note that the formula can be regarded as a special case 
of a result obtained in \cite{Erik2}. 
Here we propose an alternative proof of \eqref{eq:differential-formula}
for any convergent word $w = e_{a_1}\cdots e_{a_n} \in {\mathcal A}^0$.

The claim for $n=0$ is obvious. 
For $n = 1$, the equality 
\eqref{eq:differential-formula} is easily checked as
\[
\frac{d}{dz} L(e_{z}) = \int^{1}_{0}\frac{dt}{(t-z)^2} 
= - \frac{1}{z} + \frac{1}{z-1}.
\]

Let us assume that $n \ge 2$. 
\begin{lemma}
\label{lem: diff J}
\begin{enumerate}
\item 
For any convergent word 
$w = e_{a_1}\cdots e_{a_{n}} \in {\mathcal A}^0$, we have
\begin{eqnarray}
\nonumber
\frac{d}{dz} L(w) =
\sum_{\substack{1 \le i \le n \\ a_i = z}} 
J_{i-1}(a_{0},\dots,a_{n+1}) - \sum_{\substack{1 \le i \le n \\ a_i = z}} 
J_{i}(a_{0},\dots,a_{n+1}),
\end{eqnarray}
where, for given $0 \le i \le n$, we put 
\begin{multline*}
J_{i}(a_0,\ldots,a_{n+1})\\
:=
\begin{cases}
\, \displaystyle 
\frac{1}{a_0-a_1}\, L\left(\widehat{e_{a_{1}}}e_{a_{2}}\cdots
e_{a_{n}}\right) 
& \mbox{ if }i=0,
\\[+1.em] %
\displaystyle 
\int_{a_{0}<t_{1}<\cdots<t_{i}<t_{i+2}<\cdots<t_{n}<a_{n+1}}
\frac{dt_{i}}{(t_{i}-a_{i})(t_{i}-a_{i+1})}
\prod_{\substack{j = 1 \\ j\neq i, i+1}}^{n}
\frac{dt_{j}}{t_{j}-a_{j}} 
& \mbox{ if }1\leq i\leq n-1,
\\[+.5em] %
\displaystyle
\, -\frac{1}{a_{n}-a_{n+1}}\, 
L\left(e_{a_{1}}\cdots e_{a_{n-1}}\widehat{e_{a_{n}}}\right) 
& \mbox{ if }i=n.
\end{cases}
\end{multline*}
Here we regard the domain of the integration 
in the definition of $J_{n-1}(a_0,\dots,a_{n+1})$ as 
$\{(t_1, \dots, t_{n-1}) \in {\mathbb R}^{n-1}~|~ 
a_0 < t_1 < \cdots < t_{n-1} < a_{n+1} \}$.

\item
For $1\leq i\leq n-1$ with $a_{i} \neq a_{i+1}$, we have 
\begin{equation} \label{eq:J-i-decomp}
  J_i(a_{0},\ldots,a_{n+1})
  =\frac{1}{a_{i}-a_{i+1}}\bigl\{ 
 L(e_{a_1}\cdots\widehat{e_{a_{i+1}}}\cdots e_{a_{n}})- 
 L(e_{a_1}\cdots\widehat{e_{a_{i}}}\cdots e_{a_{n}})
 \bigr\}. 
\end{equation}
\end{enumerate}
\end{lemma}

\begin{proof}
Differentiating $L(w) = I(a_{0};a_{1},\ldots,a_{n};a_{n+1})$ 
with respect to $z$, we have
\begin{eqnarray}
\frac{d}{dz}L(w) 
& =  & 
\frac{da_1}{dz} \int_{a_0<t_2<\cdots<t_{n}<a_{n+1}} 
\biggl( \prod_{j=2}^{n} \frac{dt_j}{t_j - a_j} \biggr)
\left( \int^{t_2}_{a_0}\frac{dt_1}{(t_1 - a_1)^2} \right) 
\nonumber \\
&  & 
+ \sum_{i=2}^{n-1} \frac{da_i}{dz} 
\int_{a_0<t_1<\cdots<t_{i-1}<t_{i+1}< \cdots<t_{n}<a_{n+1}} 
\biggl(\, \prod_{\substack{j=1 \\ j \ne i}}^{n} 
\frac{dt_j}{t_j - a_j} \biggr)
\left( \int^{t_{i+1}}_{t_{i-1}} \frac{dt_i}{(t_i - a_i)^2} \right)
\nonumber \\
&  &
+ \frac{da_n}{dz} \int_{a_0<t_1< \cdots<t_{n-1}<a_{n+1}}
\biggl(\, \prod_{j=1}^{n-1} \frac{dt_j}{t_j - a_j} \biggr)
\left( \int^{a_{n+1}}_{t_{n-1}} \frac{dt_n}{(t_n - a_n)^2} \right).
\label{eq:dL-dz}
\end{eqnarray}
The first line in the right hand side of \eqref{eq:dL-dz} 
can be reduced to 
\begin{align*}
& 
\frac{da_1}{dz}
\int_{a_0<t_2<\cdots<t_{n}<a_{n+1}} 
\biggl( \prod_{j=2}^{n} \frac{dt_j}{t_j - a_j} \biggr)
\left( \int^{t_2}_{a_0}\frac{dt_1}{(t_1 - a_1)^2} \right)  \\
& = 
\frac{da_1}{dz} \int_{a_0<t_2<\cdots<t_{n}<a_{n+1}} 
\biggl( \prod_{j=2}^{n} \frac{dt_j}{t_j - a_j} \biggr)
\left( \frac{1}{a_0 - a_1} - \frac{1}{t_2 - a_1}\right)  \\
& = 
\frac{da_1}{dz} \, \bigl( 
J_{0}(a_0,\dots,a_{n+1}) - J_1(a_0,\dots,a_{n+1}) 
\bigr).
\end{align*}
By a similar computation, the second and 
the third lines in \eqref{eq:dL-dz} are reduced to 
\[
\sum_{i=1}^{n-1} \frac{da_i}{dz}  \,
\bigl( J_{i-1}(a_0,\dots, a_{n+1}) - J_i(a_0,\dots, a_{n+1}) \bigr)
\]
and 
\[
\frac{da_n}{dz} \,
\bigl( J_{n-1}(a_0,\dots, a_{n+1}) - J_n(a_0,\dots, a_{n+1}) \bigr),
\]
respectively. Thus we have proved (1). 

The claim (2) follows from the partial fraction decomposition
\[
\frac{dt_{i}}{(t_{i}-a_i)(t_{i}-a_{i+1})}=
\frac{1}{a_i-a_{i+1}}\left( 
\frac{dt_{i}}{t_{i}-a_i}-\frac{dt_{i}}{t_{i}-a_{i+1}}
\right).
\]
\end{proof}

\begin{proof}[Proof of Theorem \ref{thm:differential-formula}]
The claim (1) of Lemma \ref{lem: diff J} implies
\[
\frac{d}{dz}L(w) = S_1 + S_2, 
\]
where 
\[
S_1 = 
\sum_{\substack{1\leq i\leq n\\a_{i}=z}}
J_{i-1}(a_{0}\ldots,a_{n+1}), 
\quad 
S_2 = 
-\sum_{\substack{1\leq i\leq n\\a_{i}=z}}
J_{i}(a_0,\ldots,a_{n+1}) 
\]
We divide $S_1$ (resp., $S_2$) as 
$S_1 = S_1^{({\rm I})} + S_1^{({\rm II})}$ 
(resp., $S_2 = S_2^{({\rm I})} + S_2^{({\rm II})}$), where
\begin{eqnarray}
S_1^{({\rm I})} = 
\sum_{\substack{1\leq i\leq n\\a_{i-1}=a_{i}=z}}
J_{i-1}(a_{0},\dots,a_{n+1}), \label{eq:005} 
& &  
S_1^{({\rm II})}  = 
\sum_{\substack{1\leq i\leq n\\a_{i-1}\neq a_{i}=z}}
J_{i-1}(a_{0},\dots,a_{n+1}), \\ 
%
S_2^{({\rm I})} = 
-\sum_{\substack{1\leq i\leq n\\a_{i}=a_{i+1}=z}}
J_{i}(a_{0},\dots,a_{n+1}), \label{eq:007} 
&  &
S_2^{({\rm II})} =
-\sum_{\substack{1\leq i\leq n\\a_{i+1}\neq a_{i}=z}}
J_{i}(a_0,\dots,a_{n+1}). 
\end{eqnarray}
Since $a_{0},a_{n+1}\neq z$, 
$S_1^{({\rm I})}$ and $S_2^{({\rm I})}$
can also be expressed as 
\begin{eqnarray*}
S_1^{(\rm I)} =  
\sum_{\substack{2\leq i\leq n\\a_{i-1}=a_{i}=z}}
J_{i-1}(a_{0},\dots,a_{n+1}), \quad 
S_2^{(\rm I)} = 
-\sum_{\substack{1\leq i\leq n-1\\a_{i}=a_{i+1}=z}}
J_{i}(a_{0},\dots,a_{n+1})
\end{eqnarray*}
Therefore, $S_1^{({\rm I})} + S^{({\rm I})}_2$ vanishes identically. 
On the other hand, using \eqref{eq:J-i-decomp} we obtain
\begin{align}
S_1^{({\rm II})} & =  
\sum_{\substack{1\leq i\leq n\\a_{i-1}\neq z,a_{i}=z}}
\frac{1}{a_{i-1}-z}
L(e_{a_{1}}\cdots\widehat{e_{a_{i}}}\cdots e_{a_{n}})-
\sum_{\substack{1\leq i\leq n\\a_{i}\neq z,a_{i+1}=z}}
\frac{1}{a_{i}-z}
L(e_{a_{1}}\cdots\widehat{e_{a_{i}}}\cdots e_{a_{n}})\nonumber \\
& =\sum_{b\in\{0,1\}}\frac{1}{z-b}L\left(\sum_{1\leq i\leq n}
(-\delta_{(a_{i-1},a_{i}),(b,z)}+\delta_{(a_{i},a_{i+1}),(b,z)})
e_{a_{1}}\cdots\widehat{e_{a_{i}}}\cdots e_{a_{n}}\right).
\label{eq:009}
\end{align}
and
\begin{align}
S_2^{({\rm II})} &= 
 -\sum_{\substack{1\leq i\leq n\\a_{i-1}=z,a_{i}\neq z}}
\frac{1}{z-a_{i}}
L(e_{a_{1}}\cdots\widehat{e_{a_{i}}}\cdots e_{a_{n}})+
\sum_{\substack{1\leq i\leq n\\a_{i}=z,a_{i+1}\neq z}}
\frac{1}{z-a_{i+1}}L(e_{a_{1}}\cdots\widehat{e_{a_{i}}}\cdots e_{a_{n}})
\nonumber \\
 & =\sum_{b\in\{0,1\}}\frac{1}{z-b}L\left(\sum_{1\leq i\leq n}
 (-\delta_{(a_{i-1},a_{i}),(z,b)}+\delta_{(a_{i},a_{i+1}),(z,b)})
 e_{a_{1}}\cdots\widehat{e_{a_{i}}}\cdots e_{a_{n}}\right)
 \label{eq:010}
\end{align}
respectively, where we put
\begin{eqnarray*}
\delta_{(a,b), (x,y)} := 
\begin{cases} 
1 & \text{if $a=x$ and $b=y$} \\
0 & \text{otherwise}.
\end{cases}
\end{eqnarray*}
Therefore, since $\delta_{\{a,b\}, \{x,y \}}
 = \delta_{(a,b), (x,y)} + \delta_{(a,b), (y,x)}$
for $x \ne y$, summing up \eqref{eq:009} and \eqref{eq:010}, 
we obtain 
\[
\frac{d}{dz}L(e_{a_{1}}\cdots e_{a_{n}})=\sum_{b\in\{0,1\}}
\frac{1}{z-b}L(\partial_{z,b}(e_{a_{1}}\cdots e_{a_{n}}))
\]
which proves \eqref{eq:differential-formula} for $n \ge 2$. 
This completes the proof of Theorem \ref{thm:differential-formula}.
\end{proof}

\begin{example} \label{example:differential}
Here we show several explicit examples of the differential formula.
\begin{itemize}
\item 
For $w\in\mathbb{Q}\left\langle e_{0},e_{1}\right\rangle$
whose last letter is not $e_1$, we have 
\begin{eqnarray} \label{eq:e_a1}
\begin{cases}
\displaystyle
\frac{d}{dz}L(e_{z}e_{0}w) = -\frac{1}{z}L(e_{z}w), 
\\[+1.em]
\displaystyle
\frac{d}{dz}L(e_{z}e_{1}w) =
-\frac{1}{z-1}L(e_{z}w)
+\left(\frac{1}{z-1}-\frac{1}{z} \right)L(e_{1}w). 
\end{cases}
\end{eqnarray}
\item 
For $w\in\mathbb{Q}\left\langle e_{0},e_{z}\right\rangle$, 
we have
\begin{eqnarray}  \label{eq:e_a3}
\begin{cases}
\displaystyle 
\frac{d}{dz}L(e_{z}we_{0}) =
-\frac{1}{z}L(e_{z}w), 
\\[+1.em]
\displaystyle 
\frac{d}{dz}L(e_{1}e_{z}we_{0}) =
-\frac{1}{z}L(e_{1}e_{z}w)
+\frac{1}{z}L(e_{1}we_{0}) \\[+.3em]
\displaystyle \hspace{+8.3em} 
+\frac{1}{z-1}L((e_{z}-e_{1})we_{0}) \\[+1.em]
\displaystyle
\frac{d}{dz}L(e_{1}e_{0}we_{0}) = 
\frac{1}{z}L(e_{1}we_{0})
-\frac{1}{z}L(e_{1}e_{0}w).
\end{cases}
\end{eqnarray}
\end{itemize}

\end{example}

\section{Proof of main theorem}
\label{section:proof}

\subsection{Proof of Theorem \ref{thm:B-duality}}

This subsection is devoted to giving a proof of 
the duality formula (Theorem \ref{thm:B-duality}).
The result will be used to prove Theorem \ref{thm:sum-formula} 
in next subsection.

\begin{proof}[Proof of Thoerem \ref{thm:B-duality}]
Let $\tau$ be the anti-automorphism on $\mathcal{A}$, 
which was introduced in Section \ref{section:introduction},  
defined as
\[ 
\tau(e_0)=e_z-e_1,\ \tau(e_1)=e_z-e_0,\ \tau(e_z)=e_z.
\]
Note that ${\tau}$ can be restricted to an anti-automorphism 
on the subset ${\mathcal A}^0$ of convergent words.
Let $\gamma$ be the M\"obius transformation 
\[ 
\gamma(t) = \frac{zt-z}{t-z}
\]
which acts on ${\mathbb P}^1$ and induces 
a permutation of $\{0,1,z,\infty\}$ as
\[
\gamma(0)=1,\quad 
\gamma(1)=0,\quad 
\gamma(\infty)=z,\quad 
\gamma(z)=\infty.
\]
The M\"obius transform $\gamma$ induces 
a linear transformation on the holomorphic 
1-forms $\omega_a(t)={dt}/{(t-a)} ~~ (a\in\{0,1,z\})$ 
on $\mathbb{P}^1-\{0,1,z,\infty\}$ of the forms
\[
\omega_0(t)=\omega_1(t')-\omega_z(t'),\quad 
\omega_1(t)=\omega_0(t')-\omega_z(t'),\quad 
\omega_z(t)=-\omega_z(t') \qquad (t'=\gamma(t)).
\]
For arbitrary real number $z$ satisfying $z > 1$, we can verify that 
$\gamma$ maps the segment $[0,1] \subset {\mathbb P}^1$ to itself 
with opposite orientation. 
Thus, we obtain the desired formula 
\begin{equation} \label{eq:duality}
L(w - \tau(w)) = 0 \quad (w \in {\mathcal A}^{0})
\end{equation}
for such values of $z$. After the analytic continuation, 
we can conclude that the equality \eqref{eq:duality}
holds for any $z \in {\mathbb C} \setminus [0,1]$ 
by the identity theorem.
This completes the proof of Theorem \ref{thm:B-duality}.
\end{proof}

\begin{remark}
\begin{itemize}
\item
For any convergent word $w \in {\mathcal A}^0$, we know that 
\begin{equation} \label{eq:limit-itint}
\lim_{z \to \infty} L(w) = L(w|_{e_z = 0})
\end{equation}
holds. 
In particular, for any 
$w \in {\mathbb Q}\langle e_0, e_1 \rangle \cap {\mathcal A}^0$, 
the limit $z \to \infty$ of our duality formula \eqref{eq:duality}
is reduced to 
\begin{equation} \label{eq:duality-traditional}
\lim_{z \to \infty} L(w - \tau(w))  = 
L(w - \tau_{\infty}(w)) = 0,
\end{equation} 
where $\tau_{\infty}$ is anti-endomorphism on 
${\mathcal A}^{0}$ 
defined by
\[
\tau_{\infty}(e_0) = -e_1,\quad 
\tau_{\infty}(e_1) = -e_0, \quad
\tau_{\infty}(e_z) = 0.
\]
The equality \eqref{eq:duality-traditional} 
is nothing but the duality formula  \eqref{eq:duality-MZV} 
for multiple zeta values 
(see \cite{Hoffman} and \cite{Zagier}).

\item
Evaluating \eqref{eq:duality} at $z = -1$, we also have 
\[ 
L(w-\tau(w)) \, \bigl|_{z=-1} = 0
\quad( w\in \mathcal{A}^0).
\]
This relation is called the Broadhurst duality formula
\cite{Broadhurst} (see also \cite{Okuda}).
\end{itemize}
\end{remark}

\begin{example}
Substituting $w=e_{z}^{n}e_0e_{z}^m \ (n\ge1,m\ge0)$ 
into the theorem gives the following three term relation:
\[ 
I(0;\{z\}^n,0,\{z\}^m;1)+I(0;\{z\}^m,1,\{z\}^n;1)
-I(0;\{z\}^{n+m+1};1)=0.
\]
\end{example}

\subsection{Proof of Theorem \ref{thm:sum-formula}}

In this subsection we prove 
Theorem \ref{thm:sum-formula} (the sum formula).

\begin{proof}[Proof of Theorem \ref{thm:sum-formula}]
For integers $k \ge 2$ and $r \ge 1$ satisfying $k \ge r$, put
\begin{eqnarray*}
f_{k,r}(z) & = & \sum_{\substack{k_{1}+\cdots+k_{r}=k\\
k_{i}\ge1,~k_{r}\ge2}}
L(e_{z}e_{0}^{k_{1}-1}e_{1}e_{0}^{k_{2}-1}
\cdots e_{1}e_{0}^{k_{r}-1})\\
g_{k,r}(z) & = & L((e_{1}-e_{z})(e_{0}-e_{z})^{r-1}e_{0}^{k-r})\\
h_{k,r} & = & \sum_{\substack{k_{1}+\cdots+k_{r}=k\\
k_{i}\ge1,~k_{r}\ge2}}
L(e_{1}e_{0}^{k_{1}-1}e_{1}e_{0}^{k_{2}-1}
\cdots e_{1}e_{0}^{k_{r}-1}) 
~~ = \lim_{z \to 1 + 0} f_{k,r}(z).
\end{eqnarray*}
Theorem \ref{thm:sum-formula} is equivalent to
\begin{equation}
\label{eq:e_thm}
(-1)^{r}f_{k,r}(z)=-L(e_{1}e_{0}^{k-1})+g_{k,r}(z).
\end{equation}

We prove Theorem \ref{thm:sum-formula} by the induction on $k$. 
The case $k=2$ is obvious from the definition. 
Fix $k\geq3$ and assume that
\begin{equation} 
(-1)^{r'}f_{k-1,r'}(z)=-L(e_{1}e_{0}^{k-2})+g_{k-1,r'}(z)\qquad
(1\leq r'\leq k-1)
\label{eq:e_1}
\end{equation}
hold. 
Our goal is to prove \eqref{eq:e_thm} for any $r$ satisfying
$1 \le r \le k$ under this induction hypothesis.

The equality \eqref{eq:e_thm} for the case of $r=1$ follows 
from the definition. The case of $r=k$ follows from 
the duality formulas \eqref{eq:duality} and 
\eqref{eq:duality-traditional} 
proved in the previous subsection as follows: 
\begin{align*}
L((e_1-e_z)(e_0 - e_z)^{k-1}) 
& = (-1)^{k} L(\tau(e_1^{k-1}e_0)) \\
& = (-1)^{k} L(e_1^{k-1}e_0) \\
& = L(\tau_{\infty}(e_1 e_0^{k-1})) 
= L(e_1 e_0^{k-1}).
\end{align*}
Thus, to proceed our induction, it is enough to prove 
\eqref{eq:e_thm} for the case of $1<r<k$. In that situation, 
the differential formula \eqref{eq:differential-formula} 
shows the following. (C.f., Example \ref{example:differential}.) 

\begin{lemma} \label{lemma:derivative-fkr}
For integers $k, r$ satisfying $1<r<k$, 
the following equalities hold.
\begin{eqnarray}
\frac{d}{dz}f_{k,r}(z) & = &
-\frac{1}{z}f_{k-1,r}(z)-\frac{1}{z-1}f_{k-1,r-1}(z)
+\left( \frac{1}{z-1}-\frac{1}{z}\right) 
h_{k-1,r-1},
\label{eq:e_2} \\
\frac{d}{dz}g_{k,r}(z) 
& = & 
-\frac{1}{z}g_{k-1,r}(z)+\frac{1}{z-1}g_{k-1,r-1}(z).
\label{eq:e_3}
\end{eqnarray}
\end{lemma}

\begin{proof}[Proof of Lemma \ref{lemma:derivative-fkr}]
The equality \eqref{eq:e_2} can be easily derived from \eqref{eq:e_a1}. 
Since 
\[
g_{k,r}(z) = L(e_1e_0(e_0-e_z)^{r-2}e_0^{k-r}) - 
L(e_1e_z(e_0-e_z)^{r-2}e_0^{k-r}) - 
L(e_z(e_0-e_z)^{r-1}e_0^{k-r}), 
\]
the equality \eqref{eq:e_a3} implies \eqref{eq:e_3}.
\end{proof}

Let us take the limit  $z\to1+0$ in \eqref{eq:e_1}. 
Since $\lim_{z\to1+0} g_{k,r}(z) = 0$, we have 
\begin{equation}
(-1)^{r-1}h_{k-1,r-1}=-L(e_{1}e_{0}^{k-2}).
\label{eq:e_4}
\end{equation}
This also shows that $(-1)^{r} h_{k,r}$ is in fact independent of $r$.
From \eqref{eq:e_2}, \eqref{eq:e_3}, \eqref{eq:e_4}, 
and the assumption of the induction, we have
\[
\frac{d}{dz} \bigl( (-1)^{r}f_{k,r}(z)-g_{k,r}(z) \bigr)=0.
\]
This implies that $(-1)^{r}f_{k,r}(z)-g_{k,r}(z)$ is independent of $z$.
Since we have
\[
\lim_{z\to\infty} f_{k,r}(z) = 0, \quad
\lim_{z\to\infty} g_{k,r}(z) = L(e_{1}e_{0}^{k-1})
\]
by \eqref{eq:limit-itint}, 
we can conclude that \eqref{eq:e_thm} holds.
Thus, Theorem \ref{thm:sum-formula} is proved.
\end{proof}

Finally, we remark that the sum formula \eqref{eq:sum-MZV} 
for multiple zeta values (\cite{Granville, Zagier})
is obtained from \eqref{eq:e_thm}
after taking the limit $z \to 1 + 0$.



\begin{thebibliography}{99}


\bibitem{Broadhurst} D.J.~Broadhurst, 
{\itshape Massive 3-loop Feynman diagrams reducible 
to $SC^*$ primitives of algebras of the sixth root of unity}, 
Eur. Phys. J. C Part. Fields {\bf 8} (2) (1999) 313--333.



\bibitem{EF} B.~Enriquez, H.~Furusho, 
{\itshape Mixed Pentagon, octagon and Broadhurst duality equation}, 
Journal of Pure and Applied Algebra, Vol {\bf 216}, 
Issue 4, (2012), 982--995.

\bibitem{Euler} L.~Euler, 
{\itshape Meditationes circa singulare serierum genus}, 
Novi Comm. Acad. Sci. Petropol. {\bf 20} (1776), 140--186. 
Reprinted in Opera Omnia, ser. I, vol. {\bf 15}, 
B. G. Teubner, Berlin, 1927, pp. 217--267.



\bibitem{Goncharov} A.B.~Goncharov, 
{\itshape Galois symmetries of fundamental groupoids 
and noncommutative geometry}, 
Duke Math. J., {\bf 128}(2) (2005), 209--284.

\bibitem{Granville} A.~Granville, 
{\itshape A decomposition of Riemann's zeta-function}, 
in Analytic Number Theory, Y. Motohashi (ed.), 
London Mathematical Society Lecture Note Series {\bf 247}, 
Cambridge University Press, Cambridge, (1997) 95--101.

\bibitem{HiroseSato1} M.~Hirose, N.~Sato, 
{\itshape Iterated integrals on 
$\mathbb{P}^{1}\setminus\{0,1,\infty,z\}$  
and a class of relations among multiple zeta values}, 
in preparation. 

\bibitem{HiroseSato2} M.~Hirose, N.~Sato, 
{\itshape On Hoffman's conjectural identity}, 
in preparation. 


\bibitem{Hoffman} M.~Hoffman, 
{\itshape Multiple harmonic series}, 
Pacific J. Math. {\bf 152} (1992), 275--290.

\bibitem{Hoffman2} M.~Hoffman, 
{\itshape The algebra of multiple harmonic series}, 
J. of Alg. {\bf 194} (1997) 477--495.

\bibitem{IKZ} Y.~Ihara, M.~Kaneko and D.~Zagier, 
{\itshape Derivation and double shuffle 
relations for multiple zeta values}, 
Compos. Math. {\bf 142} (2006) 307--338.

\bibitem{Lapp-Danilevsky} J.A.~Lappo-Danilevsky, 
{\itshape Th\'{e}orie algorithmique des corps de Riemann}, 
Rec. Math. Moscou {\bf 34} (6) (1927) 113--146.

\bibitem{Okuda} J.~Okuda, 
{\itshape Duality formulas of the special values 
of multiple polylogarithms}, 
Bull. Lond. Math. Soc. {\bf 37} (2) (2005) 230--242.

\bibitem{Erik} E.~Panzer, 
{\itshape The parity theorem for multiple polylogarithms}, 
Journal of Number Theory, Vol {\bf 172}, (2017), 93--113.

\bibitem{Erik2} E.~Panzer, {\itshape Feynman integrals and hyperlogarithms}, 
PhD thesis, Humboldt-Universit\"at zu Berlin, 2014;  
arXiv:1506.07243 [math-ph].


\bibitem{Zagier} D.~Zagier, 
{\itshape Values of zeta functions and their applications}, 
in ECM volume, Progr. Math. {\bf 120} (1994), 497--512.


\end{thebibliography}
\end{document}